\documentclass[a4paper, english]{amsart}
\usepackage{amsmath}
\usepackage{amssymb}
\usepackage{enumerate}
\usepackage{verbatim}
\usepackage{microtype}

\newtheorem{MainThm}{Theorem}

\newtheorem{MainProp}[MainThm]{Proposition}

\newtheorem{thm}{Theorem}[section]
\newtheorem*{thm*}{Theorem}
\newtheorem{cor}[thm]{Corollary}
\newtheorem{lem}[thm]{Lemma}
\newtheorem{prop}[thm]{Proposition}
\newtheorem*{prop*}{Proposition}
\theoremstyle{definition}

\theoremstyle{remark}

\newtheorem{const}[thm]{Construction}
\numberwithin{equation}{section}

\newcommand{\bC}{\mathbb{C}}

\newcommand{\bL}{\mathbb{L}}

\newcommand{\bN}{\mathbb{N}}

\newcommand{\bQ}{\mathbb{Q}}
\newcommand{\bR}{\mathbb{R}}

\newcommand{\bZ}{\mathbb{Z}}

\newcommand\lra{\longrightarrow}

\newcommand\trf{\mathrm{trf}}
\newcommand\Diff{\mathrm{Diff}}

\newcommand\Sym{\mathrm{Sym}}

\usepackage[all]{xy}
\SelectTips{cm}{10}
\CompileMatrices

\title[The pseudoisotopy stable range]{An upper bound for\\ the pseudoisotopy stable range}
\author{Oscar Randal-Williams}
\email{o.randal-williams@dpmms.cam.ac.uk}
\address{Centre for Mathematical Sciences\\
Wilberforce Road\\
Cambridge CB3 0WB\\
UK}

\begin{document}
\begin{abstract}
We prove that the pseudoisotopy stable range for manifolds of dimension $2n$ can be no better than $(2n-2)$. In order to do so, we define new characteristic classes for block bundles, extending our earlier work with Ebert, and prove their non-triviality. We also explain how similar methods show that $\mathrm{Top}(2n)/\mathrm{O}(2n)$ is rationally $(4n-5)$-connected.
\end{abstract}
\maketitle

For a smooth manifold $M$, possibly with boundary, the space of smooth pseudoisotopies (also known as concordances) is  $P(M) := \Diff(M \times [0,1] \,\mathrm{rel}\, M \times \{0\})$, that is, the space of diffeomorphisms of the cylinder $M \times [0,1]$ which keep one end fixed. There is a canonical map
\begin{equation}\label{eq:ConcStab}
P(M) \lra P(M \times I)
\end{equation}
given by crossing with the interval $I$ (and unbending corners), and the (smooth) \emph{pseudoisotopy stable range} is the function
$$\phi(n) := \max\{k \in \bN \,\vert \, \text{\eqref{eq:ConcStab} is $k$-connected for all manifolds $M$ of dimension $\geq n$}\}.$$
The main theorem concerning this function is due to Igusa \cite{Igusa}, and says that
$$\phi(n) \geq \min\{\tfrac{n-7}{2}, \tfrac{n-4}{3}\}.$$
In this note we establish the following upper bound for this function.

\begin{MainThm}\label{thm:main}
$\phi(2n) \leq 2n-2$ as long as $2n \geq 6$.
\end{MainThm}

To explain our approach, let $W_{g,1} := \#^g S^n \times S^n \setminus \mathrm{int}(D^{2n})$ with $2n \geq 6$, and consider the fibration sequence
\begin{equation}\label{eq:Fibration}
\frac{\widetilde{\Diff}_\partial(W_{g,1})}{{\Diff}_\partial(W_{g,1})} \lra B\Diff_\partial(W_{g,1}) \overset{i}\lra B\widetilde{\Diff}_\partial(W_{g,1})
\end{equation}
from the classifying space of the group of diffeomorphisms of $W_{g,1}$ to the classifying space of the group of block diffeomorphisms of $W_{g,1}$. The rational cohomology of $B\Diff_\partial(W_{g,1})$ has been computed for $g \gg 0$ by Galatius and the author in \cite{GR-WActa, GR-WStab}; the rational cohomology of $B\widetilde{\Diff}_\partial(W_{g,1})$ has been computed for $g \gg 0$ by Berglund and Madsen in \cite{BM1, BM2} and in a forthcoming revision of \cite{BM2}. Ebert and the author have shown in \cite{ER-W} that the map $i$ is surjective on rational cohomology in the stable range. 

Our approach to Theorem \ref{thm:main} is motivated by forthcoming work of Berglund and Madsen, in which they show that the map induced by $i$ on rational cohomology is injective in degrees $* < 2n$ and $g \gg 0$, and more importantly for our current purpose they show that this is sharp, in the following sense.

\begin{MainProp}[Berglund--Madsen]\label{prop:Nonvanishing}
For $g \gg0$,
\begin{equation}\label{eq:BM}
\mathrm{Ker}(i^* : H^{2n}(B\widetilde{\Diff}_\partial(W_{g,1});\bQ) \to H^{2n}(B{\Diff}_\partial(W_{g,1});\bQ)) \neq 0.
\end{equation}
\end{MainProp}
This has implications for the Serre spectral sequence of \eqref{eq:Fibration}, and it is this that we shall exploit to prove Theorem \ref{thm:main}. As Proposition \ref{prop:Nonvanishing} is central to our argument, and its proof is not yet available, in Sections \ref{eq:Nonvanishing} and \ref{sec:proof} we will give an independent proof of it, which works for all $g \geq 1$ and does not require the computation of both groups. It consists of defining Mumford--Morita--Miller classes for block bundles, which extend those that we have already defined with Ebert in \cite{ER-W}, and then showing that a certain such class---namely $\tilde{\kappa}_{e^2} - \tilde{\kappa}_{p_n}$, which is easily seen to lie in the kernel \eqref{eq:BM}---is not trivial. The construction of these classes and their non-triviality may be of interest independently of Theorem \ref{thm:main}.

Finally, in Section \ref{sec:Morlet} we show how similar methods can be used to show that the space $\mathrm{Top}(2n)/\mathrm{O}(2n)$ is rationally $(4n-5)$-connected as long as $2n > 4$.

\subsection*{Acknowledgements}

I am grateful to Alexander Berglund and to Ib Madsen for (repeatedly) explaining to me their forthcoming results. The author was supported by EPSRC grant EP/M027783/1, and is grateful to Stanford University for its hospitality while this paper was written.

\section{Proof of Theorem \ref{thm:main}}

By the work of Weiss--Williams \cite[Theorem A]{WWI}, there is a certain map
\begin{equation}\label{eq:WW}
\frac{\widetilde{\Diff}_\partial(W_{g,1})}{{\Diff}_\partial(W_{g,1})} \lra \Omega^\infty(S^\infty_+ \wedge_{\bZ/2} \Omega \mathbf{Wh}^{\Diff}_s(W_{g,1}))
\end{equation}
which is $(\phi(2n)+1)$-connected. The ($\bZ/2$-)spectrum $\mathbf{Wh}^{\Diff}_s(W_{g,1})$ is the 1-connected cover of the (smooth) Whitehead spectrum $\mathbf{Wh}^{\Diff}(W_{g,1})$, which in turn is related to Waldhausen's algebraic $K$-theory of spaces by a (split) cofibre sequence of spectra
\begin{equation}\label{eq:CofibSeq}
\Sigma^\infty_+ W_{g,1} \lra \mathbf{A}(W_{g,1}) \lra \mathbf{Wh}^{\Diff}(W_{g,1}).
\end{equation}
This identification requires the stable parameterised $h$-cobordism theorem \cite{WJR}.

Our strategy is then as follows. We use a theorem of Hsiang--Staffeldt to compute the spectrum cohomology $H^*(\mathbf{Wh}^{\Diff}(W_{g,1});\bQ)$ in degrees $* \leq 2n$. We take care to compute this \emph{as a representation of the mapping class group $\Gamma_{g,1}$ of $W_{g,1}$}, in terms of the standard representation
$$H_g := H_n(W_{g,1};\bQ)$$
of $\Gamma_{g,1}$. The spectrum cohomology of $S^\infty_+ \wedge_{\bZ/2} \Omega \mathbf{Wh}^{\Diff}_s(W_{g,1})$ is then given by truncating, desuspending, and taking $\bZ/2$-invariants, and the cohomology of $\Omega^\infty(S^\infty_+ \wedge_{\bZ/2} \Omega \mathbf{Wh}^{\Diff}_s(W_{g,1}))$ is the free graded-commutative algebra on the result.

We now suppose for a contradiction that $\phi(2n) \geq 2n-1$, so the map \eqref{eq:WW} is $2n$-connected and hence we have a computation of the rational cohomology of $\frac{\widetilde{\Diff}_\partial(W_{g,1})}{{\Diff}_\partial(W_{g,1})}$ in degrees $* \leq 2n-1$, as a $\Gamma_{g,1}$-module. We then study the Serre spectral sequence for \eqref{eq:Fibration}, and derive a contradiction.

\subsection{Rational homology of the Whitehead spectrum}

We shall use Corollary 1.2 of Hsiang--Staffeldt \cite{HS}, which shows that
$$H_*(\mathbf{A}(W_{g,1});\bQ) = \pi_*(\mathbf{A}(W_{g,1})) \otimes \bQ \cong (K_*(\bZ) \otimes \bQ) \oplus (\Sigma \bar{K}_{ab})$$
where $K$ is a minimal model for the dga $C_*(\Omega W_{g,1};\bQ)$, $\bar{K}$ denotes the augmentation ideal, which inherits the structure of a a graded Lie algebra with bracket given by $[x,y] := x \cdot y - (-1)^{\vert x \vert \cdot \vert y \vert} y \cdot x$, and $\bar{K}_{ab} = \bar{K}/[\bar{K}, \bar{K}]$ is the abelianisation of this graded Lie algebra.

As $W_{g,1}$ is a suspension, the homology of $\Omega W_{g,1}$ is the tensor algebra on the vector space $H_g[n-1]$. In particular it is a free (non-commutative) algebra, so is quasi-isomorphic to $C_*(\Omega W_{g,1};\bQ)$, and we may take $K=H_*(\Omega W_{g,1};\bQ)$ with trivial differential. It follows that $\bar{K}_{ab}$ is the augmentation ideal of the free graded commutative algebra on $H_g[n-1]$, that is
$$\bar{K}_{ab} = (H_g[n-1]) \oplus \left(
\begin{aligned}
\Sym^2(H_g)[2n-2] & \text{ if $n$ is odd}\\
\wedge^2(H_g)[2n-2] & \text{ if $n$ is even}
\end{aligned}
\right) \oplus (\text{terms of degree $\geq 3n-3$}).$$
Let us write 
$$U := 
\begin{cases}
\Sym^2(H_g) & \text{ if $n$ is odd}\\
\wedge^2(H_g) & \text{ if $n$ is even.}
\end{cases}$$
Then we have
$$H_*(\mathbf{A}(W_{g,1});\bQ)  \cong (K_*(\bZ) \otimes \bQ) \oplus (H_g[n]) \oplus (U[2n-1])$$
in degrees $* \leq 2n$. Applying the cofibre sequence \eqref{eq:CofibSeq}, we obtain
$$H_*(\mathbf{Wh}^{\Diff}(W_{g,1});\bQ) \cong (\tilde{K}_*(\bZ) \otimes \bQ) \oplus (U[2n-1])$$
in degrees $* \leq 2n$. The rational homology of $\mathbf{Wh}^{\Diff}_s(W_{g,1})$ is therefore the same, as it is already 1-connected. Thus, dualising, we have
$$H^*(S^\infty_+ \wedge_{\bZ/2} \Omega \mathbf{Wh}^{\Diff}_s(W_{g,1});\bQ) \cong ((\tilde{K}_{*-1}(\bZ) \otimes \bQ) \oplus (U[2n-2]))^{\bZ/2}$$
in degrees $* \leq 2n-1$, for some involution. It follows from Farrell--Hsiang \cite{FH} (which considers the case $g=0$) that this involution acts as $-1$ on $\tilde{K}_{*-1}(\bZ) \otimes \bQ$, so this summand does not contribute to the invariants. Thus
$$H^*(S^\infty_+ \wedge_{\bZ/2} \Omega \mathbf{Wh}^{\Diff}_s(W_{g,1});\bQ) \cong (U[2n-2])^{\bZ/2}$$
in degrees $* \leq 2n-1$, for some involution on $U$. Taking the free graded-commutative algebra on this, it follows that
$$H^*(\Omega^\infty(S^\infty_+ \wedge_{\bZ/2} \Omega \mathbf{Wh}^{\Diff}_s(W_{g,1}));\bQ) \cong \bQ[0] \oplus (U[2n-2])^{\bZ/2}$$
in degrees $* \leq 2n-1$.

\subsection{The Serre spectral sequence argument}

The Serre spectral sequence for the fibration \eqref{eq:Fibration} takes the form
$$E_1^{p,q} = H^p(B\widetilde{\Diff}_\partial(W_{g,1}); H^q\left(\tfrac{\widetilde{\Diff}_\partial(W_{g,1})}{{\Diff}_\partial(W_{g,1})};\bQ\right)) \Longrightarrow H^{p+q}(B{\Diff}_\partial(W_{g,1});\bQ).$$
Under the assumption that $\phi(2n) \geq 2n-1$ we have identified the coefficients in degrees $q \leq 2n-1$, to be $\bQ$ for $q=0$ and to be $V:=U^{\bZ/2}$ for $q=2n-2$. In order for \eqref{eq:BM} to be possible, we must therefore have a non-trivial differential
$$d^{2n-1} : H^1(B\widetilde{\Diff}_\partial(W_{g,1});V) \lra H^{2n}(B\widetilde{\Diff}_\partial(W_{g,1});\bQ).$$
{In particular, the source must be non-trivial.} Note that $H^1(B\widetilde{\Diff}_\partial(W_{g,1});V)$ is a summand of $H^1(B\widetilde{\Diff}_\partial(W_{g,1});U)$, so the following will give a contradiction.

\begin{prop}
$H^1(B\widetilde{\Diff}_\partial(W_{g,1});U)=0$ for $g \gg 0$.
\end{prop}
\begin{proof} 
The action of $\Gamma_{g,1}$ on $H_n(W_{g,1};\bZ)$ preserves the intersection form, determining a homomorphism 
$$\Gamma_{g,1} \lra \begin{cases}
O_{g,g}(\bZ) & \text{ if $n$ is even}\\
Sp_{2g}(\bZ) & \text{ if $n$ is odd}.
\end{cases}
$$
This is onto if $n$ is even or $n=1,3,7$, but for the remaining odd $n$ its image is the finite-index subgroup---often denoted $\Gamma_g(1,2) \leq Sp_{2g}(\bZ)$ in the theory of theta functions---of those symplectic matrices which preserve the standard quadratic form, cf.\ \cite[Example 4.2]{BM2}. Let us write $G$ for the algebraic group $O_{g,g}$ or $Sp_{2g}$, depending on the parity of $n$, and $A_g \leq G(\bZ)$ for the image of this homomorphism. As $Sp_{2g}$ and $SO_{g,g}$ are connected semisimple algebraic groups defined over $\bQ$, it follows from a theorem of Borel--Harish-Chandra \cite[Theorem 7.8]{BHC} that $A_g$ is a lattice in $G(\bR)$, and hence by the Borel Density Theorem \cite{Borel2} that $A_g$ is Zariski dense in $G(\bR)$, so also in $G(\bC)$.

Consider the fibration sequence
$$B\widetilde{\mathfrak{Tor}}_{g,1} \lra B\widetilde{\Diff}_\partial(W_{g,1}) \lra BA_g,$$
where $B\widetilde{\mathfrak{Tor}}_{g,1}$ is defined to be the homotopy fibre. By \cite[Proposition 4.1]{BM2} we have
$$H^1(B\widetilde{\mathfrak{Tor}}_{g,1};\bQ) \cong \begin{cases}
H_g & n \equiv 3 \mod 4\\
0 & \text{else}
\end{cases}
$$
so if $n \not\equiv 3 \mod 4$ then $H^1(A_g; U) \to H^1(B\widetilde{\Diff}_\partial(W_{g,1});U)$ is an isomorphism, and if $n \equiv 3 \mod 4$ then we have an exact sequence
$$0 \lra H^1(A_g; U) \lra H^1(B\widetilde{\Diff}_\partial(W_{g,1});U) \lra (H_g \otimes U)^{A_g}.$$

In the case $n \equiv 3 \mod 4$, $n$ is odd and Zariski density of $A_g \leq Sp_{2g}(\bC)$ implies that the complexification of $(H_g \otimes \Sym^2(H_g))^{A_g}$ is $(H_g \otimes \Sym^2(H_g) \otimes \bC)^{Sp_{2g}(\bC)}$, which is contained in $(H_g^{\otimes 3} \otimes \bC)^{Sp_{2g}(\bC)}$ and so vanishes by standard invariant theory (for which we refer to \cite[\S F.2]{FultonHarris}).

It remains to show that $H^1(A_g; U)=0$.  The representation $U$ is arithmetic, so a theorem of Borel \cite[Theorem 1]{Borel} can be used to identify this with $H^1(A_g; \bQ) \otimes U^{A_g}$ as long as $g \gg 0$; see \cite[Proposition 3.9]{ER-WTorelli} for a statement of this result adapted to our situation. Hence it is enough to show the vanishing of $U^{A_g}$. 

If $n$ is odd then $U^{A_g}$ is $\Sym^2(H_g)^{A_g}$, whose complexification is the same as $\Sym^2(H_g \otimes \bC)^{Sp_{2g}(\bC)}$ by Zariski density, and this vanishes by standard invariant theory. If $n$ is even then $U^{A_g}$ is $\wedge^2(H_g)^{A_g}$, whose complexification is $\wedge^2(H_g \otimes \bC)^{O_{g,g}(\bC)}$, which also vanishes by standard invariant theory (noting that $O_{g,g}(\bC) \cong O_{2g}(\bC)$).
\end{proof}

\section{Characteristic classes of block bundles}\label{eq:Nonvanishing}

We should like to give a proof of Proposition \ref{prop:Nonvanishing}, as it does does not require the entire corpus \cite{GR-WActa, GR-WStab, BM1, BM2} and beyond to see that the kernel \eqref{eq:BM} is non-trivial. We shall show that this kernel is non-trivial by producing an explicit element in it, which will be described in terms of generalised Mumford--Morita--Miller classes. If $(\pi : E \to \vert K \vert, \mathcal{A})$ is a smooth oriented block bundle with fibre a closed $d$-manifold $M$ (we refer to \cite[Section 2]{ER-W} for this notation), in \cite[Section 3]{ER-W} Ebert and the author have associated to it
\begin{enumerate}[(i)]
\item a Leray--Serre spectral sequence $H^p(\vert K \vert, \mathcal{H}^q(M)) \Rightarrow H^{p+q}(E)$, and hence a fibre-integration map $\pi_!(-) : H^{k+d}(E) \to H^k(\vert K \vert)$,

\item a transfer map $\trf^*_\pi : H^*(E) \to H^*(\vert K \vert)$ of Becker--Gottlieb type,

\item a \emph{stable} vertical tangent bundle $T_\pi^s E \to E$,
\end{enumerate}
such that if $(\pi : E \to \vert K \vert, \mathcal{A})$ arises from a smooth fibre bundle then these data reduce to those coming from the bundle structure. In the case $d=2n$, we then employed the following ruse: If $\pi$ came from a smooth fibre bundle with $2n$-dimensional fibres, so there was an \emph{unstable} vertical tangent bundle $T_\pi E$, then we would have $e(T_\pi E)^2 = p_n(T_\pi E)$, and $\pi_!(e(T_\pi E) \cdot -) = \trf^*_\pi(-) : H^*(E) \to H^*(\vert K \vert)$. Therefore, for a monomial $p_I$ in Pontrjagin classes, if we define
$$\tilde{\kappa}_{p_I}(\pi) := \pi_!(p_I(T_\pi^s E)) \quad\quad\quad \tilde{\kappa}_{e p_I} := \trf^*_\pi(p_I(T_\pi^s E))$$
then these classes restrict to the usual $\kappa_{p_I}$ and $\kappa_{ep_I}$ on fibre bundles, and these give all generalised Mumford--Morita--Miller classes on fibre bundles.

By way of apology for this ruse, we add to the list above
\begin{enumerate}[(i)]
\setcounter{enumi}{3}
\item an Euler class $e(T_\pi E) \in H^{d}(E;\bZ)$.
\end{enumerate}
(Of course $e(T_\pi E)$ is merely notation: there is no $d$-dimensional bundle $T_\pi E$ of which it is the Euler class.) Using this Euler class, we may then define
$$\tilde{\kappa}_{e^i p_I}(\pi) := \pi_!(e(T_\pi E)^i \cdot p_I(T_\pi^s E)) \in H^*(\vert K \vert ; \bZ).$$
The symbol $\tilde{\kappa}_{ep_I}$ has the same meaning as before, by Lemma \ref{lem:Euler} (iv) below.

The existence of this Euler class is a consequence of the Fibre Inclusion Theorem of \cite{CG} (or rather its proof, which constructs a canonical such class), and the fact that the homotopy fibre of $\pi$ is homotopy equivalent to a Poincar{\'e} duality space of dimension $d$, namely $M$ \cite[Proposition 2.8]{ER-W}. As the construction is quite pretty, let us describe it.

\begin{const}\label{const:Euler}
Embed $\vert K \vert$ into $\bR^k$ for some $k \gg 0$, and let $B'$ be a closed regular neighbourhood, so that there is a retraction $r : B' \to \vert K \vert$. Let $B=D(B')$ be the double of $B'$, a closed smooth manifold. This has a retraction $s : D(B') \to B'$, and let $p : X \to B$ be the Hurewicz fibration obtained by turning $\pi$ into a fibration $\pi^f : E^f \to \vert K \vert$ and pulling it back along $r s$. As $B$ and the fibre of $p$ are Poincar{\'e} duality spaces, of dimensions $k$ and $d$ respectively, $X$ is too \cite{GottliebPoincare}, of dimension $(d+k)$. But $X \times_B X = p^*(X) \to X$ is also a fibration over a Poincar{\'e} duality space with Poincar{\'e} duality fibre, so is again a Poincar{\'e} duality space, of dimension $(2d+k)$. Writing $\Delta : X \to X \times_B X$ for the fibrewise diagonal map, which admits an umkehr map $\Delta_!$ as source and target are both Poincar{\'e}, we define
$$e(T_p X) := \Delta^* \Delta_!(1) \in H^{d}(X;\bZ).$$
We then define $e(T_\pi E)$ by restriction along $E \subset E^f \subset X\vert_{B'} \subset X$.
\end{const}

It is easy to see that the class so obtained is independent of all choices, and it is shown in \cite[\S 4]{CG} that it restricts to the Euler class on the fibre $M$. The definition given in \cite[\S 4]{CG} seems to differ by a sign, but it does not, by Lemma \ref{lem:Euler} (i) below.

\begin{lem}\label{lem:Euler}
The Euler class defined enjoys the following properties:
\begin{enumerate}[(i)]
\item If $d$ is odd then $2e(T_\pi E)=0 \in H^*(E;\bZ)$,

\item if $(\pi : E \to \vert K \vert, \mathcal{A})$ arises from a smooth fibre bundle with vertical tangent bundle $T_\pi E$, then $e(T_\pi E)$ agrees with the Euler class of the vertical tangent bundle,

\item if there is a map $r : E \to M$ such that $\pi \times r : E \to \vert K \vert \times M$ is a homotopy equivalence, then $e(T_\pi E) = r^*(e(TM))$,

\item the equation $\pi_!(e(T_\pi E) \cdot -) = \trf_\pi^*(-) : H^*(E;\bZ) \to H^*(\vert K \vert ;\bZ)$ is satisfied.
\end{enumerate}
\end{lem}

\begin{proof}
For (i), consider the involution $\tau$ of $X \times_B X$ which interchanges the two factors. When $d$ is odd, this has degree $-1$, and so $\tau^* \Delta_! = -\Delta_!$. On the other hand $\Delta^* \tau^* = \Delta^*$, so $e(T_\pi E)=-e(T_\pi E)$.

For (ii), note that if $(\pi : E \to \vert K \vert, \mathcal{A})$ arises from a smooth fibre bundle then in Construction \ref{const:Euler} we do not need to replace it by a fibration. The resulting $p : X \to B$ is a smooth fibre bundle with vertical tangent bundle $T_p X$, and the map $\Delta : X \to X \times_B X$ is a smooth embedding with normal bundle $T_p X$. Hence $\Delta^*\Delta_!(1)$ is the Euler class of $T_p X$, which restricts to the Euler class of $T_\pi E$.

For (iii), if such an $r$ exists then the fibration $p : X \to B$ admits a similar fibre homotopy trivialisation, $p \times \rho : X \overset{\sim}\to B \times M$. Then $X \times_B X \simeq B \times M \times M$ and the map $\Delta$ is given by the identity map on $B$ and the diagonal map on $M$. Hence $\Delta^*\Delta_!(1) = 1 \otimes e(TM)$.

For (iv), we must involve ourselves in the details of the construction of the transfer in \cite{CG}, with which we assume the reader is familiar. We begin by constructing a commutative diagram
\begin{equation*}
\xymatrix{
&F \ar[d] \ar@/_/[r] & W \times T^k \ar@/_/[l] \ar[d] \ar@/_/[r]& D(W) \times T^k \ar@/_/[l]\ar[d]\\
E \ar[d]^-\pi \ar@/_/[r]& X \ar@/_/[l]_-{r} \ar@/_/[r]_-{u} \ar[d]^-{p}& X' \ar@/_/[l]_-{t} \ar@/_/[r]_-{v} \ar[d]^-{p'}& X'' \ar@/_/[l]_-{s} \ar[d]^-{p''}& \\
\vert K \vert \ar@/_/[r]& B \ar@/_/[l]\ar@{=}[r]& B \ar@{=}[r]& B.
}
\end{equation*}
In this diagram, $B$ is a Poincar{\'e} duality space and $p$ is a Hurewicz fibration with fibre $F \simeq M^{d}$ (obtained as in Construction \ref{const:Euler}). $W$ is a smooth oriented manifold of dimension $(d+\ell)$ with boundary, which is homotopy equivalent to $M$, and $p'$ is a smooth fibre bundle (obtained from the Closed Fibre Smoothing Theorem of \cite{CG}). The map $p''$ is obtained as the fibrewise double of $p'$, and is a smooth oriented fibre bundle with closed fibres. Finally, the horizontal arrows express each left-hand space as a (fibrewise) retract of the the right-hand space.

For a fibration $p : S \to T$ with fibre homotopy equivalent to a finite CW complex, and a fibrewise map $f : S \to S$, let us write $\trf_{p, f}^* : H^*(S) \to H^*(T)$ for the associated transfer map. This is the map denoted $\tau^f$ in \cite{CG}. When $f = \mathrm{Id}_S$, we shorten this to $\trf_p^*$. 

By the definition of the transfer in \cite[\S 6]{CG}, we have $\trf_{p}^* = \trf_{p'', vuts}^*  s^*  t^*$. By the construction of the transfer for smooth fibre bundles in \cite[\S 5]{CG}, if we write
\begin{align*}
\delta = (\mathrm{Id}_{X''}, vuts) : X'' \lra X'' \times_B X''\\
d = (\mathrm{Id}_{X''}, \mathrm{Id}_{X''}) : X'' \lra X'' \times_B X''
\end{align*}
then we have $\trf_{p'', vuts}^*(-) = p''_!(\delta^*(d_!(1)) \cdot -)$. Thus the map $\trf_{p}^*(-)$ is $p''_!(\delta^*(d_!(1))\cdot s^*  t^*(-))=(p t s)_!(\delta^*(d_!(1))\cdot s^*  t^*(-))$, which we may write as $p_!((t  s)_!(\delta^*(d_!(1)) \cdot -)$, so we will be done if $(t  s)_!(\delta^*(d_!(1)))$ is equal to the class $e(T_{p}X)$ defined by Construction \ref{const:Euler}. Consider the homotopy pullback squares
\begin{equation*}
\xymatrix{
X'' \ar[r]^-{ts} \ar[d]^-\delta & X \ar[d]^{(\mathrm{Id} \times vu) \circ \Delta}& & X'' \ar[d]^-{ts} \ar[rr]^-{(ts \times \mathrm{Id}) \circ d}  && X \times_B X'' \ar[d]^-{\mathrm{Id} \times ts}\\
X'' \times_B X'' \ar[r]^-{ts \times \mathrm{Id}} & X \times_B X'' & & X \ar[rr]^-{\Delta} && X \times_B X
}
\end{equation*}
of Poincar{\'e} duality spaces, to which Lemma \ref{lem:Pushforward} below applies and shows that
$$(ts)_! \delta^* = \Delta^* (\mathrm{Id} \times vu)^* (ts \times \mathrm{Id})_! \quad\quad (ts \times \mathrm{Id})_! d_!(ts)^* = (\mathrm{Id} \times ts)^* \Delta_!.$$
(The signs can be determined by restricting each square to a single fibre over $B$.) Thus, writing $1=(ts)^*(1)$, we have
\begin{align*}
(ts)_!\delta^*d_!(ts)^*(1) &= \Delta^* (\mathrm{Id} \times vu)^* (ts \times \mathrm{Id})_!d_!(ts)^*(1)\\
&= \Delta^* (\mathrm{Id} \times vu)^* (\mathrm{Id} \times ts)^* \Delta_!(1) = \Delta^*\Delta_!(1)
\end{align*}
which is $e(T_{p}X)$, as required.
\end{proof}

\begin{lem}\label{lem:Pushforward}
Consider a homotopy cartesian square
\begin{equation*}
\xymatrix{
A \ar[r]^-g \ar[d]^-u& C \ar[d]^-v\\
B \ar[r]^-f& D
}
\end{equation*}
of oriented Poincar{\'e} duality spaces. Then $g_! u^* = \pm v^* f_!$.
\end{lem}

The sign ambiguity is unavoidable under the given hypotheses: changing the orientation of $B$, say, does not change $g_! u^*$, but changes $v^* f_!$ by a sign.

\begin{proof}
Let us write $a$ for the formal dimension of $A$, and so on. We assume some familiarity with the notion of Poincar{\'e} embeddings, for which we refer to \cite{Klein} for details. It is enough to prove the identity for the larger square 
\begin{equation*}
\xymatrix{
A \ar[r]^-g \ar[d]^-u& C \ar[r]^-{\mathrm{Id}_C \times \{*\}} \ar[d]^-v& C \times S^N \ar[d]^-{v \times \mathrm{Id}_{S^N}}\\
B \ar[r]^-f& D \ar[r]^-{\mathrm{Id}_D \times \{*\}}& D \times S^N.
}
\end{equation*}
By this device, we may suppose \cite[Lemma 3.1]{Klein} that $f$ admits the structure of a Poincar{\'e} embedding, with complement $K$ and normal spherical fibration $\xi$ of dimension $(d-b-1)$. Let $u^*\xi \to A$ denote the pulled back spherical fibration, and $v^* K \to C$ denote the homotopy pullback of the map $K \to D$ along $v$. There is then a homotopy commutative cube
\begin{equation*}
\xymatrix@!0{
u^*\xi \ar[dd] \ar[rd] \ar[rr] && v^*K \ar'[d][dd] \ar[rd]\\
& A \ar[dd] \ar[rr] && C \ar[dd]\\
\xi \ar'[r][rr] \ar[rd] && K \ar[rd]\\
& B \ar[rr] && D
}
\end{equation*}
in which the bottom face is homotopy cocartesian, and the vertical faces are all homotopy cartesian. It follows by Mather's Second Cube Theorem \cite[Theorem 25]{Mather} that the top face is also homotopy cocartesian. We therefore have a map 
$$C  \simeq A \cup_{u^* \xi} v^*K \lra A/u^*\xi = \mathrm{Th}(u^*\xi)$$
by collapsing $v^*K$, and similarly for $K$. This gives a homotopy commutative diagram
\begin{equation*}
\xymatrix{
C \ar[r] \ar[d]^-v & \mathrm{Th}(u^*\xi \to A) \ar[d]^-{\mathrm{Th}(u)}\\
D \ar[r] & \mathrm{Th}(\xi \to B)
}
\end{equation*}
which in cohomology yields the required equation. From this point of view, the sign ambiguity arises from the two possible choices of Thom class for $u^*\xi$: the one compatible with $[C]$ and $[A]$, or the pullback of the one compatible with $[D]$ and $[B]$.
\end{proof}

\section{Proof of Proposition \ref{prop:Nonvanishing}}\label{sec:proof}

We can extend the definition of the classes $\tilde{\kappa}_{e^i p_I}$ to block bundles having fibres $W_{g,1}$ by filling in a disc in each fibre, giving a new block bundle with fibre $W_g := W_{g,1} \cup_\partial D^{2n} = \#^g S^n \times S^n$. There are therefore defined universal characteristic classes $\tilde{\kappa}_{e^i p_I} \in H^*(B\widetilde{\Diff}_\partial(W_{g,1});\bQ)$, by the proof of \cite[Theorem 3.4]{ER-W}.

In particular, we have a class $\tilde{\kappa}_{e^2} - \tilde{\kappa}_{p_n} \in H^{2n}(B\widetilde{\Diff}_\partial(W_{g,1});\bQ)$ which vanishes in $H^{2n}(B\Diff_\partial(W_{g,1});\bQ)$, because $e^2 = p_n$ on the total space of a smooth fibre bundle. Proposition \ref{prop:Nonvanishing} is an immediate consequence of the following.

\begin{prop}\label{prop:Nonvanishing2}
For each $g \geq 1$ and each $n \geq 3$ there is a block bundle $(\pi : E \to \vert K \vert, \mathcal{A})$ with fibre $W_{g,1}$, such that
\begin{enumerate}[(i)]
\item $\tilde{\kappa}_{e^2}(\pi)=0 \in H^{2n}(\vert K \vert;\bQ)$,

\item $\tilde{\kappa}_{p_n}(\pi) \neq 0 \in H^{2n}(\vert K \vert;\bQ)$.
\end{enumerate}
Therefore $\tilde{\kappa}_{e^2} - \tilde{\kappa}_{p_n} \neq 0 \in H^{2n}(B\widetilde{\Diff}_\partial(W_{g,1});\bQ)$.
\end{prop}
\begin{proof}
From Lemma \ref{lem:Euler} (iii) it follows that the $\tilde{\kappa}_{e^i}$ vanish for all $i>0$ on all fibre homotopically trivial block bundles. We will therefore construct $\pi$ to be fibre homotopically trivial, guaranteeing that $\tilde{\kappa}_{e^2}(\pi)=0$. 

We will use the (space-level) surgery fibration of Quinn \cite{Quinn}, which following the discussion in \cite[Section 3.2]{BM1}, in particular equation (43), may be put in the form
$$\left(\frac{\mathrm{hAut}_\partial(W_{g,1})}{\widetilde{\Diff}_\partial(W_{g,1})}\right)_{(1)} \lra \mathrm{map}_*(W_{g,1}/\partial W_{g,1}, G/O)_{(1)} \overset{\sigma}\lra \bL_{2n}(\bZ)_{(1)}.$$
Thus to construct a fibre homotopically trivial block bundle over $B$ (with some triangulation) it is enough to give a map $f : B \to \mathrm{map}_*(W_{g,1}/\partial W_{g,1}, G/O)_{(1)}$ and a nullhomotopy of $\sigma \circ f$.

For simplicity of exposition we restrict to the case $n=2k$. We let $B=S^n \times S^n$, write $a, b \in H^n(B;\bQ)$ for a hyperbolic basis, and write $e_1, f_1, \ldots, e_g, f_g \in H^n(W_{g,1}, \partial W_{g,1};\bQ)$ for a hyperbolic basis. Write the $n$th Hirzebruch $L$-polynomial as $\mathcal{L}_n = A p_n + B p_{n/2}^2$ modulo other Pontrjagin classes, for some constants $A$ and $B$. It is well-known that $A \neq 0$, and less well-known but true \cite[Lemma A.1]{Weiss} that $B \neq 0$. 

As the composition
$$p: G/O \overset{i}\lra BO \overset{\prod p_i}\lra \prod_{i=1}^\infty K(\bZ, 4i)$$
has homotopy fibre with finite homotopy groups, we claim that may find a map $f$ whose adjoint $\hat{f} : (B \times W_{g,1}, B \times \partial W_{g,1}) \to (G/O, *)$ composed with $i$ gives a class
$$\xi \in KO^0(B \times W_{g,1}, B \times \partial W_{g,1})$$
which has $p_{n/2}(\xi) = C \cdot (a \otimes e_1 + b \otimes f_1)$, $p_n(\xi) = -\tfrac{2B C^2}{A}\cdot a\cdot b \otimes e_1 \cdot f_1$, and all other rational Pontrjagin classes zero, for some constant $C \neq 0$. To establish this claim, let the map
$$\varphi: (B \times W_{g,1}, B \times \partial W_{g,1}) \lra \left(\prod_{i=1}^\infty K(\bZ, 4i), *\right)$$
classify the pair of relative cohomology classes $L \cdot (a \otimes e_1 + b \otimes f_1)$ and $-\tfrac{2B L^2}{A}\cdot a\cdot b \otimes e_1 \cdot f_1$, for some integer $L \neq 0$ large enough that these classes are integral. For each $N > 0$ consider the map $\phi_N : \prod_{i} K(\bZ, 4i) \to \prod_{i} K(\bZ, 4i)$ which multiplies by $N^i$ on $K(\bZ, 4i)$. The diagram
\begin{equation*}
\xymatrix{
B \times \partial W_{g,1} \ar@{^(->}[d] \ar[r]^-{\varphi\vert_{B \times \partial W_{g,1}}}& {*} \ar[d] \ar[r] & G/O \ar[d]^-p \\
B \times W_{g,1} \ar@{-->}[rru]_-{\hat{f}} \ar[r]^-\varphi& {\prod\limits_{i=1}^\infty K(\bZ, 4i)} \ar[r]^-{\phi_N}& {\prod\limits_{i=1}^\infty K(\bZ, 4i)}
}
\end{equation*}
then admits a dotted lift $\hat{f}$ for $N$ large enough, as the universal obstructions to finding such a lift lie in the cohomology of $\prod_{i} K(\bZ, 4i)$ with finite coefficients, and are therefore annihilated (on each skeleton) by some $\phi_N$. The resulting map $\hat{f}$ gives $p_{n/2}(\xi) = L \cdot N^{n/2} \cdot (a \otimes e_1 + b \otimes f_1)$, $p_n(\xi) = -\tfrac{2B L^2}{A} \cdot N^n \cdot a\cdot b \otimes e_1 \cdot f_1$, and all other Pontrjagin classes zero, as required (with $C=L \cdot N^{n/2}$).

We must show that the composition
$$B=S^n \times S^n \overset{f}\lra \mathrm{map}_*(W_{g,1}/\partial W_{g,1}, G/O)_{(1)} \overset{\sigma}\lra \bL_{2n}(\bZ)_{(1)}$$
is nullhomotopic, but we shall allow ourselves to precompose $f$ with self-maps $k_N : S^n \times S^n \to S^n \times S^n$ having degree $N \neq 0$ on both factors (such a precomposition preserves the form of Pontrjagin classes which has been arranged above). With this in mind, it is enough to show that
$$\sigma \circ f = 0 \in [B, \bL_{2n}(\bZ)] \otimes \bQ.$$
This group may be identified with $H^{4*}(B;\bQ)$. If $n \equiv 0 \mod 4$ then the component of degree $n=2k=4\ell$ is identified with the K{\"u}nneth factor of
$$\tfrac{1}{8}\mathcal{L}_{3\ell}(\xi) \in H^{12\ell}(B \times W_{g,1}, B \times \partial W_{g,1};\bQ)$$
in $H^{4\ell}(B;\bQ) \cong H^{4\ell}(B;\bQ) \otimes H^{8\ell}(W_{g,1}, \partial W_{g,1};\bQ)$. But $\mathcal{L}_{3\ell}(\xi)=0$ by observation, as only $p_{4\ell}(\xi)$ and $p_{2\ell}(\xi)$ are non-zero. Whatever the class of $n$ modulo 4, the component of degree $2n=4k$ is identified with the K{\"u}nneth factor of
$$\tfrac{1}{8}\mathcal{L}_{2k}(\xi) \in H^{8k}(B \times W_{g,1}, B \times \partial W_{g,1};\bQ)$$
in $H^{4k}(B;\bQ)$. But by construction
$$\mathcal{L}_{2k}(\xi)= A \cdot (-\tfrac{2B C^2}{A}\cdot a\cdot b \otimes e_1 \cdot f_1) + B \cdot (C \cdot (a \otimes e_1 + b \otimes f_1))^2=0.$$

We therefore obtain a map $f$, with $\sigma \circ f$ nullhomotopic and $i \circ \hat{f}$ classifying a vector bundle $\xi'$ having $p_{n/2}(\xi') = D \cdot (a \otimes e_1 + b \otimes f_1)$, $p_n(\xi') = -\tfrac{2B D^2}{A}\cdot a\cdot b \otimes e_1 \cdot f_1$, and all other Pontrjagin classes zero, for some constant $D \neq 0$. (The constant will have changed when we precomposed the original choice of $f$ with the maps $k_N$.) The associated block bundle $\pi : E \to \vert K \vert \approx B$ has $T_v^s E \simeq_s TE - \pi^*TB = \epsilon^{2n}+\xi'$ (see \cite[Lemma 3.3]{ER-W}) and so
$$\tilde{\kappa}_{p_n}(\pi) = \pi_!(p_n(T_v^s E)) = \pi_!(p_n(\xi')) = -\tfrac{2B D^2}{A}\cdot a\cdot b \neq 0$$
as required.

It is not difficult to adapt the above argument to work for $n = 2k+1$. The essential point is that if we write $\mathcal{L}_n = A p_n + B p_{\tfrac{n-1}{2}}p_{\tfrac{n+1}{2}}$ modulo all other Pontrjagin classes, then $A \neq 0$ and again by \cite[Lemma A.1]{Weiss} $B \neq 0$. We then take $B=S^{2k-1} \times S^{2k+3}$ and proceed as above.
\end{proof}

\section{Rational connectivity of $\mathrm{Top}(2n)/\mathrm{O}(2n)$}\label{sec:Morlet}

Our goal in this section is to show how similar techniques to those we have been using imply the following. 

\begin{thm}\label{thm:Morlet}
$\pi_*(B\Diff_\partial(D^{2n})) \otimes \bQ=0$ for $1 \leq * \leq 2n-5$.
\end{thm}

This extends the analogous calculation of Farrell--Hsiang \cite{FH}, which established the same result in degrees $1 \leq * \leq \phi(2n)$. By smoothing theory we have a homotopy equivalence
$B\Diff_\partial(D^{2n}) \simeq \Omega^{2n}_0 (\mathrm{Top}(2n)/\mathrm{O}(2n))$ as long as $2n > 4$, from which we deduce that

\begin{cor}
$\mathrm{Top}(2n)/\mathrm{O}(2n)$ is rationally $(4n-5)$-connected as long as $n>2$.
\end{cor}

On the other hand, it has been shown by Weiss \cite{Weiss} that
$$H^{4n}(B\mathrm{Top}(2n);\bQ) \lra H^{4n}(B\mathrm{O}(2n);\bQ)$$
has nontrivial kernel for $n \gg0$ (namely, the class $e^2 - p_n$), so $\mathrm{Top}(2n)/\mathrm{O}(2n)$ is \emph{not} rationally $(4n-1)$-connected.

\begin{proof}[Proof of Theorem \ref{thm:Morlet}]
Let $W_g = \#^g S^n \times S^n \setminus \mathrm{int}(D^{2n})$, with $2n \geq 6$, and choose a collar $[0,1) \times \partial W_{g,1} \subset W_{g,1}$ and a disc $D^{2n} \subset (0,1) \times \partial W_{g,1}$. The map 
\begin{equation}\label{eq:qwerty}
\frac{\widetilde{\Diff}_\partial(D^{2n})}{{\Diff}_\partial(D^{2n})} \lra \frac{\widetilde{\Diff}_\partial(W_{g,1})}{{\Diff}_\partial(W_{g,1})}
\end{equation}
is $(2n-4)$-connected, by Morlet's lemma of disjunction \cite[Corollary 3.2]{BLR}. Consider the fibration
$$\frac{\widetilde{\Diff}_\partial(W_{g,1})}{{\Diff}_\partial(W_{g,1})} \lra B{\Diff}_\partial(W_{g,1}) \overset{i}\lra \widetilde{\Diff}_\partial(W_{g,1}).$$
Up to isotopy any diffeomorphism $\varphi$ representing an element of $\pi_1(B\widetilde{\Diff}_\partial(W_{g,1}))$ may be supposed to be equal to the identity on the collar: the map \eqref{eq:qwerty} is then preserved by that induced by $\varphi$, and it then follows that $\varphi$ acts trivially on $H^*\left(\tfrac{\widetilde{\Diff}_\partial(W_{g,1})}{{\Diff}_\partial(W_{g,1})};\bQ\right)$ in the range of degrees $* \leq 2n-4$ where the map \eqref{eq:qwerty} is a cohomology injection. Hence the Serre spectral sequence
$$H^p(B\widetilde{\Diff}_\partial(W_{g,1}); H^q\left(\tfrac{\widetilde{\Diff}_\partial(W_{g,1})}{{\Diff}_\partial(W_{g,1})};\bQ\right)) \Longrightarrow H^{p+q}(B\Diff_\partial(W_{g,1});\bQ)$$
associated to this fibration has a product structure in this range of degrees. But Berglund--Madsen have shown that the map 
$$i^* : H^*(B\widetilde{\Diff}_\partial(W_{g,1});\bQ) \lra H^*(B{\Diff}_\partial(W_{g,1});\bQ)$$
is an isomorphism in degrees $* \leq 2n-1$ as long as $g \gg 0$. This result will appear soon in a revision of \cite{BM2}. It follows that $H^q\left(\tfrac{\widetilde{\Diff}_\partial(W_{g,1})}{{\Diff}_\partial(W_{g,1})};\bQ\right)=0$ for $1 \leq q \leq 2n-4$, and hence by the map \eqref{eq:qwerty} that $H^q\left(\tfrac{\widetilde{\Diff}_\partial(D^{2n})}{{\Diff}_\partial(D^{2n})};\bQ\right)=0$ for $1 \leq q \leq 2n-5$.

On the other hand, the surgery fibration sequence shows that $\frac{\mathrm{hAut}_\partial(D^{2n})}{\widetilde{\Diff}_\partial(D^{2n})}$ is rationally acyclic, and $\mathrm{hAut}_\partial(D^{2n}) \simeq *$ by the Alexander trick, so $B\widetilde{\Diff}_\partial(D^{2n})$ is rationally acyclic. Boundary connect-sum makes this into an $H$-space, so it has trivial rational homotopy groups. Thus $\widetilde{\Diff}_\partial(D^{2n})$ has finitely-many components, and each one is rationally acyclic. The group ${\Diff}_\partial(D^{2n})$ has the same components, and so the quotient $\tfrac{\widetilde{\Diff}_\partial(D^{2n})}{{\Diff}_\partial(D^{2n})}$ is rationally homotopy equivalent to $B{\Diff}_\partial(D^{2n})_0$, the classifying space of the component of the identity in $B{\Diff}_\partial(D^{2n})$. It follows from the above that its rational cohomology, and hence rational homotopy, vanishes in degrees $1 \leq * \leq 2n-5$.
\end{proof}

\bibliographystyle{plain}
\bibliography{MainBib}  

\end{document}